\documentclass[11pt]{amsart}
\usepackage{amsmath}
\usepackage{amssymb}
\usepackage{amsfonts}
\usepackage{amsthm}
\usepackage{enumerate}
\usepackage{hyperref}
\usepackage{color}
\usepackage{amscd}
\usepackage{verbatim}
\usepackage{calc}

\numberwithin{equation}{section}
\theoremstyle{plain}
\newtheorem{thm}{Theorem}[section]

\newtheorem{prop}[thm]{Proposition}
\newtheorem{lem}[thm]{Lemma}

\newtheorem*{ques*}{Question}

\theoremstyle{definition}
\newtheorem{defn}[thm]{Definition}
\newtheorem{nota}[thm]{Notation}

\newtheorem{exmp}[thm]{Example}
\newtheorem{rem}[thm]{Remark}
\newtheorem{fact}[thm]{Fact}

\newtheorem{conj}[thm]{Conjecture}

\begin{document}


\title[Chains of semidualizing modules]
{Chains of semidualizing modules}


     \author[E. Amanzadeh]{Ensiyeh Amanzadeh}

\address{School of Mathematics, Institute for Research in
Fundamental Sciences (IPM),  P.O. Box: 19395-5746,  Tehran,  Iran. }

\email{en.amanzadeh@gmail.com}





\begin{abstract}
Let $(R, \mathfrak{m}, k)$ be a commutative Noetherian local ring.
We study the suitable chains of semidualizing $R$-modules.
We prove  that when $R$ is Artinian, the existence of
a suitable chain of semidualizing modules of length $n=\mathrm{max}\,\{\,i\geqslant 0\ |\  \mathfrak{m}^{i}\neq 0\,\}$
implies that the the Poincar$\acute{\mathrm{e}}$ series of $k$ and the  Bass series of $R$ have  very specific forms.
Also, in this case we show that the Bass numbers of $R$ are strictly increasing. This gives an insight into
the question of Huneke about the Bass numbers of $R$.
\end{abstract}


\maketitle

\section{Introduction} \label{sec1}
In this paper $R$ is a commutative Noetherian local ring with maximal ideal $\mathfrak{m}$ and residue field $k$.
Foxby \cite{f},  Vasconcelos \cite{v} and Golod \cite{g} independently initiated the study of
semidualizing modules.
A finite (i.e. finitely generated) $R$-module $C$ is called \emph{semidualizing} if the natural
homothety map  $\chi_C^R: R \longrightarrow \mathrm{Hom}_R(C,  C)$ is an isomorphism
and $\mathrm{Ext}^{\geqslant1}_R(C,  C)=0$ (see \cite[Definition 1.1]{hj}).
Trivial examples of semidualizing $R$-modules are $R$ itself and a dualizing $R$-module when one exists.
The set of all isomorphism classes of semidualizing $R$-modules is denoted by $\mathfrak{G}_0(R)$,
and the isomorphism class of a semidualizing $R$-module $C$ is denoted $[C]$.
The set $\mathfrak{G}_0(R)$ has a rich structure, for instance, it comes equipped with
an ordering based on the notion total reflexivity.
For semidualizing $R$-modules $B$ and $C$, we write $[C] \trianglelefteq [B]$ whenever
$B$ is totally $C$-reflexive.
In \cite{gerko}, Gerko defines chains in $\mathfrak{G}_0(R)$. A \emph{chain} in
$\mathfrak{G}_0(R)$ is a sequence
$[C_n]\trianglelefteq \cdots \trianglelefteq[C_1] \trianglelefteq [C_0]$,
and such a chain has length $n$ if $[C_i]\neq[C_j]$ whenever $i\neq j$.
In \cite{s-w1}, Sather-Wagstaff uses the length of chains in $\mathfrak{G}_0(R)$
to provide a lower bound for the cardinality of $\mathfrak{G}_0(R)$.

The well-known result of Foxby \cite{f}, Reiten \cite{reiten} and Sharp \cite{s} is that
$R$ admits a dualizing module if and only if $R$ is Cohen-Macaulay
and a homomorphic image of a local Gorenstein ring.
Then Jorgensen et.\! al. \cite{jls-w}, characterize the Cohen-Macaulay local rings which admit
dualizing modules and non-trivial semidualizing modules.
Recently, Amanzadeh and Dibaei \cite{ad},  characterize the Cohen-Macaulay local rings which admit
 dualizing modules and  suitable chains of semidualizing modules.

In \cite{gerko}, Gerko has shown that, when $R$ is Artinian, the existence of
collections of strongly $\mathrm{Tor}$-independent semidualizing $R$-modules
(this condition conjecturally equivalent to the existence of a corresponding long
chain of semidualizing modules)
implies that the Poincar$\acute{\mathrm{e}}$ series of $k$ has a very specific form.
We show that the \cite[Conjecture 4.6]{gerko} holds true for suitable chains
in $\mathfrak{G}_0(R)$ (see Theorem~\ref{T1}). This result implies that
the set $\mathfrak{G}_0(R)$ can not contain suitable chains of arbitrary
length. Indeed, the length of a suitable chain in $\mathfrak{G}_0(R)$ is less than the number
$l=\mathrm{min}\,\{\,i> 0\ |\  \mathfrak{m}^{i}=0\,\}$.

There are several works which were motivated by the following questions of Huneke about
the Bass numbers $\mu^i_R(R)=\mathrm{rank}_k(\mathrm{Ext}^i_R(k, R))$.
For instance, see \cite{bs-wy}, \cite{csv}, \cite{jl} and \cite{s-w3}.
However, each of the following questions is still open in general.

\begin{ques*}
Let $R$ be a Cohen-Macaulay ring.
\begin{itemize}
\item[(a)]If the sequence $\{\mu^i_R(R)\}$ is bounded above by a polynomial in $i$, must $R$ be Gorenstein?
\item[(b)]If $R$ is not Gorenstein, must the sequence $\{\mu^i_R(R)\}$ grow exponentially?
\end{itemize}
\end{ques*}

In \cite{s-w3}, Sather-Wagstaff has proved that if there exists a chain of semidualizing
$R$-modules of length $n$, then the Bass series of $R$
is a product of $n$ power series with positive integer coefficients, and the Bass numbers
of $R$ are bounded below by a polynomial of degree $n-1$.
In Theorem~\ref{T2}, we prove that an Artinian ring $R$ with $\mathfrak{m}^{n+1}= 0$ is $SD(n)$-full
if and only if there is a suitable chain in $\mathfrak{G}_0(R)$ of length $n$.
Then by this result and some results of \cite{gerko}, we show that the existence of a
suitable chain in $\mathfrak{G}_0(R)$ implies that the Bass series of $R$ has very
specific form and the sequence $\{\mu^i_R(R)\}$
is strictly increasing (see Proposition~\ref{Bass series} and Remark~\ref{remark Bass}).

\section{Preliminaries} \label{sec2}

This section contains definitions and background material.
\begin{defn}\label{reflexive} (\cite[Definition 2.7]{hj} and \cite[Theorem 5.2.3 and Definition 6.1.2]{s-w2})
Let $C$ be a semidualizing $R$-module. A finite $R$-module $M$ is
\emph{totally} $C$-\emph{reflexive} when it satisfies the following conditions\,:
\begin{itemize}
\item[(i)] the natural homomorphism
$\delta^C_M: M \longrightarrow \mathrm{Hom}_R(\mathrm{Hom}_R(M, C), C)$
is an isomorphism, and
\item[(ii)]
$\mathrm{Ext}^{\geqslant1}_R(M,  C)=0=\mathrm{Ext}^{\geqslant1}_R(\mathrm{Hom}_R(M, C),  C)$.
\end{itemize}
\end{defn}

\begin{fact}
Define the order $\trianglelefteq$ on $\mathfrak{G}_0(R)$. For $[B], [C]\in \mathfrak{G}_0(R)$,
write  $[C]\trianglelefteq[B]$ when $B$ is totally $C$-reflexive (see, e.g., \cite{s-w1}).
This relation is reflexive and antisymmetric \cite[Lemma 3.2]{fs-w1},
but it is not known whether it is transitive in general.
Also, write $[C]\vartriangleleft[B]$ when $[C]\trianglelefteq [B]$ and $[C]\neq[B]$.

In the case  $D$ is a dualizing $R$-module, one has $[D]\trianglelefteq [B]$
for any semidualizing $R$-module $B$, by \cite[(V.2.1)]{h}.

If $[C]\trianglelefteq[B]$ then $\mathrm{Hom}_R(B, C)$ is a semidualizing and
$[C]\trianglelefteq[\mathrm{Hom}_R(B, C)]$ (\cite[Theorem 2.11]{c2}).
\end{fact}

The following notations are taken from \cite{s-w1}.

\begin{nota}\label{B_i}
Assume that  $[C_n]\vartriangleleft \cdots \vartriangleleft[C_1] \vartriangleleft [C_0]$
is a chain in $\mathfrak{G}_0(R)$.
For each $i\in [n]=\{1,\cdots, n\}$ set $B_i=\mathrm{Hom}_R(C_{i-1}, C_i)$.
For each sequence of integers $\mathbf{i} = \{i_1, \cdots , i_j\}$  with $j \geqslant 1$
and $1\leqslant i_1 <  \cdots  < i_j \leqslant n$, set
$B_{\mathbf{i}}=B_{i_1}\otimes_R\cdots \otimes_R B_{i_j}$.
($B_{\{i_1\}} =B_{i_1}$ and set $B_\emptyset=C_0$).
\end{nota}

\begin{fact}\label{f1}
If $[C_n]\vartriangleleft \cdots \vartriangleleft[C_1] \vartriangleleft [C_0]$
is a chain in $\mathfrak{G}_0(R)$ such that $C_0\cong R$, then by \cite[Corollary 3.3]{gerko}, for each
$i\in[n]$ we have the following isomorphisms
$$C_i\cong C_0 \otimes_R \mathrm{Hom}_R(C_0,  C_1) \otimes_R \cdots \otimes_R\mathrm{Hom}_R(C_{i-1},  C_i)
\cong B_1\otimes_R \cdots \otimes_R B_i\,.$$
\end{fact}

For a semidualizing $R$-module $C$, set ${(-)}^{\dag_C}=\mathrm{Hom}_R(-,  C)$.

\begin{defn}\label{suitable}(\cite[Remark 3.4]{s-w1} and \cite[Definition 3.1]{ad})
Let  $[C_n]\vartriangleleft \cdots \vartriangleleft[C_1] \vartriangleleft [C_0]$ be
a chain in $\mathfrak{G}_0(R)$ of length $n$.
For each sequence of integers $\mathbf{i} = \{i_1, \cdots , i_j\}$ such that
$j \geqslant 0$ and $1\leqslant i_1 <  \cdots  < i_j \leqslant n$, set
$C_{\mathbf{i}}= C_0^{\dag_{C_{i_1}}\dag_{C_{i_2}}\cdots\dag_{C_{i_j}}}$.
(When $j = 0$, set  $C_{\mathbf{i}} = C_\emptyset= C_0$).

We say that the above chain is \emph{suitable} if  $C_0=R$ and
$C_{\mathbf{i}}$ is totally $C_t$-reflexive, for all
$\mathbf{i}$ and $t$ with $i_j\leqslant t \leqslant n$.

Note that if $[C_n]\vartriangleleft \cdots \vartriangleleft[C_1] \vartriangleleft [R]$
is a suitable chain, then $C_{\mathbf{i}}$ is a semidualizing $R$-module for each
$\mathbf{i}\subseteq [n]$.
\end{defn}

\begin{prop}\label{P1}
Assume that  $[C_n]\vartriangleleft \cdots \vartriangleleft[C_1] \vartriangleleft [C_0]$
is a suitable chain in $\mathfrak{G}_0(R)$.
\begin{itemize}
\item[(1)] \emph{(\cite[Lemma 4.5]{s-w1} and \cite[Remark 3.3]{ad})} For each sequence $\mathbf{i}=\{i_1, \cdots , i_j\}\subseteq [n]$,
the $R$-module $B_{\mathbf{i}}$ is a semidualizing.
\item[(2)] \emph{(\cite[Lemma 4.6]{s-w1} and \cite[Remark 3.3]{ad})} If $\mathbf{i}, \mathbf{s}\subseteq [n]$
are two sequences with $\mathbf{s}\subseteq \mathbf{i}$, then
$[B_{\mathbf{i}}] \trianglelefteq[B_{\mathbf{s}}]$ and
$\mathrm{Hom}_R(B_{\mathbf{s}}, B_{\mathbf{i}})\cong B_{\mathbf{i}\setminus \mathbf{s}}$.
\item[(3)] \emph{(\cite[Theorem 3.3]{s-w1} and \cite[Remark 3.3]{ad})}
$|\mathfrak{G}_0(R)|\geqslant |\big\{ [C_{\mathbf{i}}]\ |\ \mathbf{i}\subseteq [n]\big\}|=2^n$.
\item[(4)] \emph{(\cite[Theorem 4.7]{s-w1} and \cite[Remark 3.3]{ad})}
$\big\{ [B_{\mathbf{u}}]\ |\ \mathbf{u}\subseteq [n]\big\}=\big\{ [C_{\mathbf{i}}]\ |\ \mathbf{i}\subseteq [n]\big\}$.
\end{itemize}
\end{prop}
For an $R$-module $M$, the $i$th {\it Bass number} of $M$ is the integer
$\mu^i_R(M)=\mathrm{rank}_k(\mathrm{Ext}^i_R(k, M))$, and the {\it Bass series}
of $M$ is the formal Laurent series $I^M_R(t)=\Sigma_{i\in \mathbb{Z}} \mu^i_R(M)t^i$.
The $i$th {\it Betti number} of $M$ is the integer
$\beta_i^R(M)=\mathrm{rank}_k(\mathrm{Ext}^i_R(M, k))=\mathrm{rank}_k(\mathrm{Tor}^R_i(k, M))$,
and the {\it Poincar$\acute{e}$ series} of $M$ is the formal Laurent series
$P_M^R(t)=\Sigma_{i\in \mathbb{Z}} \beta_i^R(M)t^i$.

\section{Results} \label{sec3}
First we investigate suitable chains modulo regular sequences.
\begin{prop}\label{P2}
Assume that $\mathbf{x}=x_1,\cdots, x_d$ is an $R$-regular sequence and
$[C_n]\vartriangleleft \cdots \vartriangleleft[C_1] \vartriangleleft [C_0]$
is a suitable chain in $\mathfrak{G}_0(R)$ of length $n$. Then
$[\overline{C}_n]\vartriangleleft \cdots \vartriangleleft[\overline{C}_1] \vartriangleleft [\overline{C}_0]$
is also a suitable chain in $\mathfrak{G}_0(\overline{R})$ of length $n$, where
$\overline{R}=R/\mathbf{x}R$ and $\overline{C}_i=\overline{R}\otimes_R C_i$ for $i=0,1,\cdots,n$.
\end{prop}
\begin{proof}
By \cite[Theorem 5.1]{c2}, $\overline{C}_i$ is a semidualizing $\overline{R}$-module for all $i=0,\cdots, n$.
Also, by \cite[Theorem 5.10]{c2} and \cite[Proposition 4.2.18]{s-w2},
$[\overline{C}_i] \vartriangleleft [\overline{C}_{i-1}]$ for $i=1,\cdots, n$.
Therefore
$[\overline{C}_n]\vartriangleleft \cdots \vartriangleleft[\overline{C}_1] \vartriangleleft [\overline{C}_0]$
is a chain in $\mathfrak{G}_0(\overline{R})$ of length $n$.
Now we show that this chain is suitable.

For a semidualizing $\overline{R}$-module $C$, set ${(-)}^{\overline{\dag}_C}=\mathrm{Hom}_{\overline{R}}(-,  C)$.
For each sequence of integers $\mathbf{i} = \{i_1, \cdots , i_j\}$ such that
$j \geqslant 0$ and $1\leqslant i_1 <  \cdots  < i_j \leqslant n$, set
$\mathbf{C}_{\mathbf{i}}= \overline{C}_0^{\overline{\dag}_{\overline{C}_{i_1}}\overline{\dag}_{\overline{C}_{i_2}}\cdots\overline{\dag}_{\overline{C}_{i_j}}}$.
(When $j = 0$, set  $\mathbf{C}_{\mathbf{i}} = \mathbf{C}_\emptyset= \overline{C}_0=\overline{R}$).

By induction on $j$ we prove that
$$\mathbf{C}_{\mathbf{i}}=
\overline{C}_0^{\overline{\dag}_{\overline{C}_{i_1}}\overline{\dag}_{\overline{C}_{i_2}}\cdots
\overline{\dag}_{\overline{C}_{i_j}}}
\cong \overline{ C_0^{\dag_{C_{i_1}}\dag_{C_{i_2}}\cdots\dag_{C_{i_j}}}}
=\overline{R}\otimes_R C_0^{\dag_{C_{i_1}}\dag_{C_{i_2}}\cdots\dag_{C_{i_j}}}=\overline{C_{\mathbf{i}}}$$
and $\mathbf{C}_{\mathbf{i}}$ is totally $\overline{C}_t$-reflexive for all $\mathbf{i}$
and $t$ with $i_j\leqslant t \leqslant n$.

When $j=0$, there is nothing to prove. If $j=1$, then
$\mathbf{C}_{\mathbf{i}}=\mathrm{Hom}_{\overline{R}}(\overline{C}_0, \overline{C}_{i_1})\cong \overline{C_0^{\dag_{C_{i_1}}}}$.
As $[C_n]\vartriangleleft \cdots \vartriangleleft[C_1] \vartriangleleft [C_0]$ is a suitable chain,
$C_0^{\dag_{C_{i_1}}}$ is totally $C_t$-reflexive for all $i_1\leqslant t \leqslant n$.
Hence $\mathbf{C}_{\mathbf{i}}$ is totally $\overline{C}_t$-reflexive for all
$i_1\leqslant t \leqslant n$, by  \cite[Theorem 5.10]{c2}.

Let $j>1$. We have
$\mathbf{C}_{\mathbf{i}}= (\overline{C}_0^{\overline{\dag}_{\overline{C}_{i_1}}\overline{\dag}_{\overline{C}_{i_2}}\cdots
\overline{\dag}_{\overline{C}_{i_{j-1}}}})^{\overline{\dag}_{\overline{C}_{i_j}}}$
and by induction
$$\overline{C}_0^{\overline{\dag}_{\overline{C}_{i_1}}\overline{\dag}_{\overline{C}_{i_2}}\cdots
\overline{\dag}_{\overline{C}_{i_{j-1}}}}
\cong \overline{ C_0^{\dag_{C_{i_1}}\dag_{C_{i_2}}\cdots\dag_{C_{i_{j-1}}}}}$$
is totally $\overline{C}_t$-reflexive for all
$i_{j-1}\leqslant t \leqslant n$.
As $C_0^{\dag_{C_{i_1}}\dag_{C_{i_2}}\cdots\dag_{C_{i_{j-1}}}}$ is totally $C_{i_j}$-reflexive, we obtain
the isomorphism
$$\mathrm{Hom}_{\overline{R}}(\overline{ C_0^{\dag_{C_{i_1}}\dag_{C_{i_2}}\cdots\dag_{C_{i_{j-1}}}}}, \overline{C}_{i_j})
\cong \overline{\mathrm{Hom}_R(C_0^{\dag_{C_{i_1}}\dag_{C_{i_2}}\cdots\dag_{C_{i_{j-1}}}}, C_{i_j})}\,.$$
Therefore $\mathbf{C}_{\mathbf{i}}\cong \overline{ C_0^{\dag_{C_{i_1}}\dag_{C_{i_2}}\cdots\dag_{C_{i_j}}}}=\overline{C}_{\mathbf{i}}$
and $\mathbf{C}_{\mathbf{i}}$ is totally $\overline{C}_t$-reflexive for all $i_j\leqslant t \leqslant n$,
by \cite[Theorem 5.10]{c2}.
Thus
$[\overline{C}_n]\vartriangleleft \cdots \vartriangleleft[\overline{C}_1] \vartriangleleft [\overline{C}_0]$
is a suitable chain in $\mathfrak{G}_0(\overline{R})$.
\end{proof}

For the remaining part of this paper we assume that $(R, \mathfrak{m}, k)$ is an Artinian local ring
and that all modules are finite.

\begin{defn}\cite[Definitions 4.1 and 4.2]{gerko}
The modules $K_1, K_2, \cdots, K_n$ are said to be weakly $\mathrm{Tor}$-independent if
$\mathrm{amp}(\otimes^L_{1\leqslant i \leqslant n} K_i)=0$.
These modules are said to be strongly $\mathrm{Tor}$-independent if for any subset $\Lambda \subseteq [n]$
we have $\mathrm{amp}(\otimes^L_{ i \in \Lambda} K_i)=0$.
\end{defn}

In the case $n=2$ both notions are equivalent to the classical
$\mathrm{Tor}$-independence, i.e., to the condition that $\mathrm{Tor}^R_{>0}(K_1, K_2)=0$.

\begin{thm}\label{gerko} \cite[Theorem 4.5]{gerko}
If the modules $K_1, K_2, \cdots, K_n$ are non-free and strongly
$\mathrm{Tor}$-independent, then  $\mathfrak{m}^n\neq 0$. If, under the same conditions,
$\mathfrak{m}^{n+1}= 0$, then the Poincar$\acute{e}$ series of $k$ has the form
$1/\prod^n_{i=1} (1-d_i t)$ for some positive integers $d_i$.
\end{thm}

\begin{conj}\label{conj}  \cite[Conjecture 4.6]{gerko}
If $[C_n]\vartriangleleft \cdots \vartriangleleft[C_1] \vartriangleleft [C_0]$
is a chain in $\mathfrak{G}_0(R)$ of length $n$,
then $\mathfrak{m}^n\neq 0$. If, under the same conditions, $\mathfrak{m}^{n+1}= 0$,
then the Poincar$\acute{\mathrm{e}}$ series of $k$ has the form
$1/\prod^n_{i=1} (1-d_i t)$ for some positive integers $d_i$.
\end{conj}

In the next result we prove the conjecture for suitable chains.

\begin{thm} \label{T1}
Let $[C_n]\vartriangleleft \cdots \vartriangleleft[C_1] \vartriangleleft [C_0]$
be a suitable chain in $\mathfrak{G}_0(R)$ of length $n$,
then $\mathfrak{m}^n\neq 0$. If, under the same conditions, $\mathfrak{m}^{n+1}= 0$,
then the Poincar$\acute{e}$ series of $k$ has the form
$1/\prod^n_{i=1} (1-d_i t)$ for some positive integers $d_i$.
\end{thm}
\begin{proof}
The non-free semidualizing modules $B_1, B_2, \cdots, B_n$ are  strongly
$\mathrm{Tor}$-independent, where $B_i$ is as in Notation~\ref{B_i}, for each $i\in [n]$.
Indeed, by Proposition~\ref{P1}, for each sequence of integers
$\mathbf{i} = \{i_1, \cdots , i_j\}$
with $j \geqslant 1$ and $1\leqslant i_1 <  \cdots  < i_j \leqslant n$,
$B_{\mathbf{i}}=B_{i_1}\otimes_R\cdots \otimes_R B_{i_j}$ is a semidualizing $R$-module.
Thus $\mathrm{amp}(B_{i_1}\otimes^L_R\cdots \otimes^L_R B_{i_j})=\mathrm{amp}(B_{\mathbf{i}})=0$.
Therefore the assertion concludes by Theorem~\ref{gerko}.
\end{proof}

Note that Theorem~\ref{T1} implies \cite[Theorem 4.8]{gerko}. Indeed, for $n\leqslant 3$ any chain
of the form $[D]=[C_n]\vartriangleleft \cdots \vartriangleleft[C_1] \vartriangleleft [C_0]=[R]$, where
$D$ is  dualizing,   is also a suitable chain of length $n$.
At this moment we do not know whether a chain for which
Conjecture~\ref{conj} holds true, is a suitable chain.
However we guess that, if $[C_n]\vartriangleleft \cdots \vartriangleleft[C_1] \vartriangleleft [C_0]$
is a chain in $\mathfrak{G}_0(R)$ such that
the semidualizing modules $B_1, \cdots, B_n$ are strongly $\mathrm{Tor}$-independent,
then this chain is suitable.

\begin{rem}
When $(S, \mathfrak{n})$ is a Cohen-Macaulay local ring with dimension $d$ and
$[C_n]\vartriangleleft \cdots \vartriangleleft[C_1] \vartriangleleft [C_0]$
is a suitable chain in $\mathfrak{G}_0(S)$ of length $n$, then, by Proposition~\ref{P2},
$[\overline{C}_n]\vartriangleleft \cdots \vartriangleleft[\overline{C}_1] \vartriangleleft [\overline{C}_0]$
is also a suitable chain in $\mathfrak{G}_0(\overline{S})$ of length $n$, where
$\overline{S}=S/\mathbf{x}S$, $\overline{C}_i=\overline{S}\otimes_S C_i$ and $\mathbf{x}=x_1,\cdots, x_d$ is an
$S$-regular sequence. 
If $\overline{\mathfrak{n}}^{n+1}=0$, then by Theorem~\ref{T1}, the Poincar$\acute{\mathrm{e}}$ series of
$S/\mathfrak{n}\cong \overline{S}/\overline{\mathfrak{n}}$ has the form
$1/\prod^n_{i=1} (1-d_i t)$ for some positive integers $d_i$.

Also, one may see that the set $\mathfrak{G}_0(S)$ can not contain suitable chains of arbitrary
length. Indeed, the length of a suitable chain in $\mathfrak{G}_0(S)$ is less than the number
$l=\mathrm{min}\,\{\,i> 0\ |\  \overline{\mathfrak{n}}^{i}=0\,\}$.
In \cite{s-w3}, some other upper bounds for the length of chains of semidualizing modules are given.
\end{rem}

\begin{defn} \cite[Definition 4.9]{gerko}
An Artinian ring $R$ is called $SD(n)$-full if the following conditions are satisfied.
\begin{itemize}
\item[(i)] $\mathfrak{m}^{n+1}= 0$.
\item[(ii)] There are strongly Tor-independent non-free semidualizing modules \\
$K_1, K_2, \cdots, K_n$  such that for any subset $\Lambda \subseteq [n]$ the module
$\otimes_{i\in \Lambda} K_i$ \\
is semidualizing.
\end{itemize}
\end{defn}

\begin{thm}\label{T2}
An Artinian ring $R$ is $SD(n)$-full if and only if there is a suitable chain
in $\mathfrak{G}_0(R)$ of length $n$ and $\mathfrak{m}^{n+1}= 0$.
\end{thm}

\begin{proof}
First we assume that $\mathfrak{m}^{n+1}= 0$ and
$[C_n]\vartriangleleft \cdots \vartriangleleft[C_1] \vartriangleleft [C_0]$
is a suitable chain in $\mathfrak{G}_0(R)$ of length $n$.
Then, by Proposition~\ref{P1},
the non-free semidualizing modules $B_1, B_2, \cdots, B_n$ are  strongly
$\mathrm{Tor}$-independent and,
for each subset $\mathbf{i}=\{i_1, \cdots , i_j\}\subseteq [n]$,
$B_{\mathbf{i}}=B_{i_1}\otimes_R\cdots \otimes_R B_{i_j}$ is a semidualizing $R$-module.
Therefore $R$ is $SD(n)$-full.

For the converse, assume that  $K_1, K_2, \cdots, K_n$ are strongly $\mathrm{Tor}$-independent
non-free semidualizing modules such that for any subset $\Lambda \subseteq [n]$ the module
$\otimes_{i\in \Lambda} K_i$ is semidualizing.
Set $C_0=R$ and $C_j=\otimes_{1\leqslant i\leqslant j} K_i$ for each $j\in [n]$.
By \cite[Proposition 3.5]{gerko}, it can be seen that the semidualizing modules
$C_0,C_1,\cdots, C_n$ form a chain in $\mathfrak{G}_0(R)$ and then
the sequence
\begin{equation}\label{e1}
[C_n]\vartriangleleft [C_{n-1}]\vartriangleleft\cdots \vartriangleleft[C_1] \vartriangleleft [C_0]
\end{equation}
is a chain in $\mathfrak{G}_0(R)$ of length $n$.
We show that (\ref{e1}) is a suitable chain.
For each sequence $\mathbf{i}=\{i_1, \cdots , i_j\}\subseteq [n]$
such that $j \geqslant 0$ and $1\leqslant i_1 <  \cdots  < i_j \leqslant n$, we have
$C_{\mathbf{i}}= C_0^{\dag_{C_{i_1}}\dag_{C_{i_2}}\cdots\dag_{C_{i_j}}}$.
(When $j = 0$, set  $C_{\mathbf{i}} = C_\emptyset= C_0$ ).
Note that for each $t$, $i_j\leqslant t\leqslant n$,
$$C_t=\otimes_{1\leqslant l\leqslant t} K_l=(\otimes_{1\leqslant l\leqslant i_1} K_l)
\otimes_R (\otimes_{i_1< l\leqslant i_2} K_l)\otimes_R \cdots \otimes_R(\otimes_{i_{j-1}< l\leqslant i_j} K_l)
 \otimes_R (\otimes_{i_j< l\leqslant t} K_l).$$
On the other hand, by \cite[Proposition 3.5]{gerko}, we have
$$C_0^{\dag_{C_{i_1}}\dag_{C_{i_2}}}=\mathrm{Hom}_R(\otimes_{1\leqslant l\leqslant i_1} K_l,
\otimes_{1\leqslant l\leqslant i_2} K_l) \cong \otimes_{i_1< l\leqslant i_2} K_l$$
and so
$$C_0^{\dag_{C_{i_1}}\dag_{C_{i_2}}\dag_{C_{i_3}}}
=\mathrm{Hom}_R(\otimes_{i_1< l\leqslant i_2} K_l, \otimes_{1\leqslant l\leqslant i_3} K_l)
\cong (\otimes_{1\leqslant l\leqslant i_1} K_l)\otimes_R(\otimes_{i_2< l\leqslant i_3} K_l).$$
By proceeding in this way and using \cite[Proposition 3.5]{gerko}, one may see that
$C_{\mathbf{i}}$ is totally $C_t$-reflexive. Thus the sequence (\ref{e1}) is a suitable chain.
\end{proof}

\begin{exmp}
A ring with $\mathfrak{m}^{n+1}= 0$ is $SD(n)$-full if and only if
there is a suitable chain in $\mathfrak{G}_0(R)$ of length $n$.

Let $F$ be a field. Set $S_i=F\ltimes F^{a_i}$ for all $1\leqslant i\leqslant n$, where $a_i>1$.
Then $S_i$ is an Artinian ring with dualizing module $D_i=\mathrm{Hom}_F(S_i, F)$.
As $\mathrm{type}\ S_i=a_i\neq 1$, the ring $S_i$ is not Gorenstein.
Set $S=\otimes_F^{1\leqslant i \leqslant n} S_i$.
By \cite[Example 4.11]{gerko}, $S$ is $SD(n)$-full.
Now, we construct a suitable chain of length $n$ in $\mathfrak{G}_0(S)$.
If $K_i$ and $M_i$ are $S_i$-modules such that $K_i$ is finite, then there is $S$-module isomorphism
$\mathrm{Hom}_{S}(\otimes_F^{1\leqslant i \leqslant n} K_i, \otimes_F^{1\leqslant i \leqslant n} M_i)
\cong \otimes_F^{1\leqslant i \leqslant n}\mathrm{Hom}_{S_i}(K_i, M_i)$, by \cite[Proposition A.1.5]{s-w2}.
When $K_i$ is semidualizing, by \cite[Proposition 2.3.6]{s-w2}, $\otimes_F^{1\leqslant i \leqslant n} K_i$
is also a semidualizing $S$-module, and $\otimes_F^{1\leqslant i \leqslant n} M_i$ is
totally $\otimes_F^{1\leqslant i \leqslant n} K_i$-reflexive by \cite[Proposition 5.3.3]{s-w2},
whenever $M_i$ is totally $K_i$-reflexive.
Note that if we take, for each $i$, a module  $K_i\in \{S_i, D_i\}$, then
the module  $\otimes_F^{1\leqslant i \leqslant n} K_i$ is semidualizing $S$-module.
Set $C_0=S$ and
$C_j=(\otimes_F^{1\leqslant i \leqslant j} D_i)\otimes_F(\otimes_F^{j< i \leqslant n} S_i)$
for all $1\leqslant j\leqslant n$.
Therefore we obtain the chain
$[C_n]\vartriangleleft \cdots \vartriangleleft[C_1] \vartriangleleft [C_0]$
in $\mathfrak{G}_0(S)$ of length $n$. We show that this chain is suitable.
Assume that $\mathbf{u}=\{u_1, \cdots, u_j\}$ is a sequence of integers such that
$1\leqslant u_1<\cdots < u_j\leqslant n$ and
$C_{\mathbf{u}}=C_0^{\dag_{C_{u_1}}\cdots \dag_{C_{u_j}}}$.
We have $C_0^{\dag_{C_{u_1}}}\cong C_{u_1}$,
$$C_0^{\dag_{C_{u_1}} \dag_{C_{u_2}}}\cong \mathrm{Hom}_S(C_{u_1}, C_{u_2})
\cong (\otimes_F^{1\leqslant i \leqslant u_1} S_i)\otimes_F
(\otimes_F^{u_1< i \leqslant u_2} D_i)\otimes_F
(\otimes_F^{u_2< i \leqslant n} S_i)\,,\  \mathrm{and}$$
$$\begin{array}{llll}
C_0^{\dag_{C_{u_1}} \dag_{C_{u_2}}\dag_{C_{u_3}}}& = &
\mathrm{Hom}_S(C_0^{\dag_{C_{u_1}} \dag_{C_{u_2}}}, C_{u_3})\\
 & \cong  &
(\otimes_F^{1\leqslant i \leqslant u_1} D_i)\otimes_F
(\otimes_F^{u_1< i \leqslant u_2} S_i)\otimes_F
(\otimes_F^{u_2< i \leqslant u_3} D_i)\otimes_F
(\otimes_F^{u_3< i \leqslant n} S_i).
\end{array}$$
By proceeding in this way one obtains the following isomorphism
{\small $$C_{\mathbf{u}}\cong\left\{ \begin{array}{ll}
\hspace{-0.25cm}(\otimes_F^{1\leqslant i \leqslant u_1} S_i)\otimes_F
(\otimes_F^{u_1< i \leqslant u_2} D_i)\otimes_F\cdots \otimes_F
(\otimes_F^{u_{j-1}< i \leqslant u_j} D_i)\otimes_F
(\otimes_F^{u_j< i \leqslant n} S_i) & \mathrm{if}\ j\ \mathrm{is\ even}, \\
  &   \\
\hspace{-0.25cm}(\otimes_F^{1\leqslant i \leqslant u_1} D_i)\otimes_F
(\otimes_F^{u_1< i \leqslant u_2} S_i)\otimes_F\cdots \otimes_F
(\otimes_F^{u_{j-1}< i \leqslant u_j} D_i)\otimes_F
(\otimes_F^{u_j< i \leqslant n} S_i)  &  \mathrm{if}\ j\ \mathrm{is\ odd}.
\end{array} \right.$$ }
As $C_t=(\otimes_F^{1\leqslant i \leqslant t} D_i)\otimes_F(\otimes_F^{t< i \leqslant n} S_i)$,
one can see that $C_{\mathbf{u}}$ is totally $C_t$-reflexive for all $u_j\leqslant t\leqslant n$.
Hence $[C_n]\vartriangleleft \cdots \vartriangleleft[C_1] \vartriangleleft [C_0]$
is a suitable chain in $\mathfrak{G}_0(S)$.

Note that
$B_j=\mathrm{Hom}_S(C_{j-1}, C_j)\cong S_1\otimes_F \cdots
\otimes_F S_{j-1}\otimes_F D_j \otimes_F S_{j+1} \otimes_F \cdots \otimes_F S_n$
for all $1\leqslant j\leqslant n$.
The semidualizing modules $B_1, \cdots, B_n$ are strongly $\mathrm{Tor}$-independent
and for any subset $\Lambda \subseteq [n]$ the module
$\otimes_{i\in \Lambda} B_i$ is semidualizing.
\end{exmp}

The next result follows from \cite[Proposition 5.4]{gerko} and the proof of Theorem~\ref{T2}.

\begin{lem}\label{Bass series 1}
Let $R$ be an Artinian ring with $\mathfrak{m}^{n+1}=0$.
Assume that
$[C_n]\vartriangleleft \cdots \vartriangleleft[C_1] \vartriangleleft [C_0]$
is a suitable chain in $\mathfrak{G}_0(R)$ of length $n$.
Then for each $i\in[n]$ the Poincar$\acute{e}$ series of $B_i$ is
$P_{B_i}^R(t)= (\beta_0^R(B_i)-t)/(1-\beta_0^R(B_i)t)$
and the Bass series of $B_{[n]\backslash i}$ is
$I^{B_{[n]\backslash i}}_R(t)= (\mu^0_R(B_{[n]\backslash i})-t)/(1-\mu^0_R(B_{[n]\backslash i})t)$.
\end{lem}

\begin{prop}\label{Bass series}
Let $R$ be an Artinian ring with $\mathfrak{m}^{n+1}=0$.
Assume that
$[C_n]\vartriangleleft \cdots \vartriangleleft[C_1] \vartriangleleft [C_0]$
is a suitable chain in $\mathfrak{G}_0(R)$ of length $n$.
Then for each $i\in[n]$ the Poincar$\acute{e}$ series of $C_i$ is
$$P_{C_i}^R(t)= {\prod_{j=1}^i (\beta_0^R(B_j)-t)\over\prod_{j=1}^i(1-\beta_0^R(B_j)t)}\ ,$$
$I^{C_n}_R(t)=1$, and for $i\neq n$, the Bass series of $C_i$ is
$$I^{C_i}_R(t)= {\prod_{j=i+1}^n (\beta_0^R(B_j)-t)\over\prod_{j=i+1}^n(1-\beta_0^R(B_j)t)}
={\prod_{j=i+1}^n (\mu^0_R(B_{[n]\backslash j})-t)\over\prod_{j=i+1}^n(1-\mu^0_R(B_{[n]\backslash j})t)}\ .$$
\end{prop}
\begin{proof}
As $\mathfrak{m}^{n+1}=0$, this suitable chain is of maximum length, by Theorem~\ref{T1}.
Thus the module $C_n$ is the dualizing $R$-module. Hence the Bass series of $C_n$ is $I^{C_n}_R(t)=1$.
For $i\in[n]$, we have $C_i\cong B_1\otimes_R\cdots\otimes_R B_i$, by Fact~\ref{f1}.
Therefore $P_{C_i}^R(t)=P_{B_1}^R(t)P_{B_2}^R(t)\cdots P_{B_i}^R(t)$, by \cite[Lemma 1.5.3]{af}.
Hence, by Lemma~\ref{Bass series 1}, one gets
$$P_{C_i}^R(t)=({\beta_0^R(B_1)-t\over 1-\beta_0^R(B_1)t})({\beta_0^R(B_2)-t\over 1-\beta_0^R(B_2)t})
\cdots ({\beta_0^R(B_i)-t\over 1-\beta_0^R(B_i)t})\,.$$
Also, we have $C_i\cong B_1\otimes_R\cdots\otimes_R B_i\cong \mathrm{Hom}_R(B_{\{i+1,\cdots, n\}}, B_{[n]})$,
by Proposition~\ref{P1}. Thus, by \cite[Lemma 1.5.3]{af}, one gets
$$I^{C_i}_R(t)=P_{B_{\{i+1,\cdots, n\}}}^R(t) I_R^{B_{[n]}}(t)
=P_{B_{i+1}}^R(t)\cdots P_{B_{n}}^R(t) I_R^{B_{[n]}}(t)\,.$$
As $B_{[n]}=B_1\otimes_R\cdots\otimes_R B_n\cong C_n$,
one has $I_R^{B_{[n]}}(t)=I_R^{C_n}(t)=1$.
Therefore, by Lemma~\ref{Bass series 1},
$$I^{C_i}_R(t)=P_{B_{i+1}}^R(t)\cdots P_{B_{n}}^R(t)=({\beta_0^R(B_{i+1})-t\over 1-\beta_0^R(B_{i+1})t})
\cdots ({\beta_0^R(B_n)-t\over 1-\beta_0^R(B_n)t})\,.$$
By Proposition~\ref{P1},  $B_{[n]\backslash j}\cong\mathrm{Hom}_R(B_j, B_{[n]})$
and so $I_R^{B_{[n]\backslash j}}(t)=P_{B_j}^R(t)I_R^{B_{[n]}}(t)=P_{B_j}^R(t)$.
Note that $\mu^0(B_{[n]\backslash j})=\beta_0(B_j)$ and then the final equality follows.
\end{proof}

From Proposition~\ref{Bass series}, the Bass series of $R$ is
$$I_R(t)=I_R^{\mathrm{Hom}_R(C_1, C_1)}(t)=P_{C_1}^R(t)I_R^{C_1}(t)
=I^{B_{[n]\backslash 1}}_R(t)I_R^{C_1}(t)
={\prod_{j=1}^n (\mu^0_R(B_{[n]\backslash j})-t)\over\prod_{j=1}^n(1-\mu^0_R(B_{[n]\backslash j})t)}\,, \ \mathrm{and}$$
$$I_R(t)=P_{C_n}^R(t)
={\prod_{j=1}^n (\beta_0^R(B_j)-t)\over\prod_{j=1}^n(1-\beta_0^R(B_j)t)}\ .$$
For each $i\in[n]$,
$I^{B_{[n]\backslash i}}_R(t)=\alpha_i+(\alpha_i^2-1)t+ \alpha_i(\alpha_i^2-1)t^2+\alpha_i^2(\alpha_i^2-1)t^3+\cdots$,
where $\alpha_i=\mu^0_R(B_{[n]\backslash i})$ (see \cite[Propositions 5.1, 5.2 and 5.4]{gerko}).
Thus $\{\mu^j_R(B_{[n]\backslash i})\}$ is strictly increasing
(indeed, it has exponential growth, since
$\mu^{j}_R(B_{[n]\backslash i})\geqslant \alpha_i^j$ for all $j\geqslant 1$).
Therefore $\{\mu^j_R(R)\}$ is also strictly increasing whenever $n\geqslant 1$, since
$I_R(t)=I^{B_{[n]\backslash 1}}_R(t)I^{B_{[n]\backslash 2}}_R(t)\cdots I^{B_{[n]\backslash n}}_R(t)$.
Also, similar to the proof of \cite[Theorem 3.5]{s-w3}, it is easy to see that
$\{\mu^j_R(R)\}$ is bounded below by a polynomial in $j$ of degree $n-1$.

\begin{rem}\label{remark Bass}
Assume that $(S, \mathfrak{n})$ is a Cohen-Macaulay local ring with dimension $d$ and $\mathbf{x}=x_1,\cdots, x_d$
is an $S$-regular sequence. Set $\overline{S}=S/\mathbf{x}S$.
As $\mathrm{Ext}_S^i(S/\mathfrak{n}, S)\cong \mathrm{Ext}_{\overline{S}}^{i-d}(S/\mathfrak{n}, \overline{S})$
for all $i\geqslant d$, we have $\mu^i_S(S)=\mu^{i-d}_{\overline{S}}(\overline{S})$,
for all $i\geqslant d$. Thus we get $I_S(t)=t^d I_{\overline{S}}(t)$.
Now, if $[C_n]\vartriangleleft \cdots \vartriangleleft[C_1] \vartriangleleft [C_0]$
is a suitable chain in $\mathfrak{G}_0(S)$ of length $n$ and
$\overline{\mathfrak{n}}^{n+1}=0$, then, by Proposition~\ref{P2} and
the above discussion,  $I_S(t)$ has a very specific form and $\{\mu^i_S(S) \}$
is strictly increasing.
\end{rem}



\end{document}